\documentclass[12pt]{article}
\usepackage[margin=1in]{geometry}
\usepackage{amsmath, amsthm, amssymb}
\usepackage{tikz}
\usetikzlibrary{arrows,automata}
\usepackage[pdfa]{hyperref}

\renewcommand{\wr}{\circ}

\newtheorem{thm}{Theorem}[section]
\newtheorem{cor}[thm]{Corollary}
\newtheorem{lem}[thm]{Lemma}
\newtheorem{prop}[thm]{Proposition}
\newtheorem{conj}[thm]{Conjecture}

\theoremstyle{definition}

\input{pics}

\allowdisplaybreaks

\renewcommand{\L}{\mathcal{L}}

\tikzstyle{vertex}=[draw,thick,fill=white,circle,inner sep=2pt]
\tikzstyle{full}=[draw,thick,fill=black,circle,inner sep=2pt]
\tikzstyle{empty}=[draw,color=black!20!white,thick,fill=white,circle,inner sep=2pt]

\title{Hadamard diagonalizable graphs of order at most $36$}
\author{Jane Breen\footnote{Ontario Tech University, Oshawa, Ontario, Canada \texttt{jane.breen@ontariotechu.ca}} \and 
Steve Butler\footnote{Iowa State University, Ames, Iowa, USA \texttt{\{butler,lidicky,awnr,sysong\}@iastate.edu}. Butler was supported in part by a Simons Foundation grant (\#427264).
 Lidick\'y supported in part by NSF grant DMS-1855653.
Riasanovsky supported in part by NSF grant DMS-1839918 (RTG).} \and 
Melissa Fuentes\footnote{University of Delaware, Newark, Delaware, USA \texttt{melfue@udel.edu}} \and
Bernard Lidick\'y\footnotemark[2] \and
Michael Phillips\footnote{University of Colorado Denver, Denver, Colorado, USA \texttt{michael.2.phillips@ucdenver.edu}} \and
Alexander W.~N.~Riasanovsky\footnotemark[2] \and
Sung-Yell Song\footnotemark[2] \and 
Ralihe R. Villagr\'{a}n\footnote{Departamento de Matem\'{a}ticas, Centro de Investigaci\'{o}n y de Estudios Avanzados del IPN, Mexico City, Mexico \texttt{rvillagran@math.cinvestav.mx}. Supported in part by CONACyT.} \and
Cedar Wiseman \footnote{University of Wyoming, Laramie, Wyoming, USA \texttt{cwisema3@uwyo.edu}} \and
Xiaohong Zhang \footnote{University of Waterloo, Waterloo, Ontario, Canada \texttt{xiaohong.zhang@uwaterloo.ca}}
}
\date{\today}

\newcommand{\oururl}{\url{http://lidicky.name/pub/hadamard/}}

\begin{document}

\maketitle

\begin{abstract}
If the Laplacian matrix of a graph has a full set of orthogonal eigenvectors with entries $\pm1$, then the matrix formed by taking the columns as the eigenvectors is a Hadamard matrix and the graph is said to be Hadamard diagonalizable.

In this article, we prove that if $n=8k+4$ the only possible Hadamard diagonalizable graphs are $K_n$, $K_{n/2,n/2}$, $2K_{n/2}$, and $nK_1$, and we develop an efficient computation for determining all graphs diagonalized by a given Hadamard matrix of any order. Using these two tools, we determine and present all Hadamard diagonalizable graphs up to order 36. Note that it is not even known how many Hadamard matrices there are of order 36.
\end{abstract}

\textbf{Keywords:} Hadamard matrix, Laplacian matrix, Cayley graph, graph product, experimental mathematics. \\

\textbf{AMS Mathematical Subject Classification:} 05C50 (15B34, 05B20, 05C76, 05C85) \\
\section{Introduction}
A real Hadamard matrix is an $n\times n$ matrix $H$ with entries in $\pm1$ with the property that $H^TH=nI$; in other words, the columns of $H$ are orthogonal. These matrices have been extensively studied, and it is known that a necessary condition for the existence of such a matrix is that $n=1,2$, or is a multiple of $4$.  A well-known and still open problem concerns the question of whether this is sufficient.

\begin{conj}
Hadamard matrices exist for all orders $n$ of the form $n=4k$.
\end{conj}

Examples of Hadamard matrices of order $2^k$ were constructed by Sylvester in 1867. Defining $H_0=[1]$, we have
\[H_{k+1} = \left[\begin{array}{rr} H_k & H_k \\ H_k & -H_k\end{array}\right].\]
For example,
\[H_1 = \left[\begin{array}{rr} 1 & 1 \\ 1 & -1 \end{array}\right], \quad H_2 = \left[\begin{array}{rrrr} 1 & 1 & 1 & 1\\ 1 & -1 & 1 & -1 \\ 1 & 1 & -1 & -1 \\ 1 & -1 & -1 & 1 \end{array}\right].\]

Two Hadamard matrices are said to be \emph{equivalent} if we can produce one from the other by some combination of the following operations: permuting rows, permuting columns, negating some subset of rows, negating some subset of columns. A \emph{normalized} Hadamard matrix is one which has every entry in the first row and column equal to $+1$. It is easily seen that every Hadamard matrix is equivalent to a normalized Hadamard matrix, by negating combinations of rows and columns.

A graph $G$ is defined in terms of a set of vertices $V(G)$, and a set of edges $E(G)$ which consist of pairs of vertices. The vertices $u$ and $v$ are said to be \emph{adjacent} if there is an edge $\{u,v\} \in E(G)$. The degree of a vertex $u$, denoted $\deg(u)$, is the number of vertices adjacent to $u$. A graph $G$ is said to be \emph{regular} if all the degrees of the vertices of $G$ are equal. 

Given a graph $G$, the Laplacian matrix $L$ is defined entrywise by
\[
L_{uv}=\begin{cases}\deg(u)&\text{if }u=v,\\
-1&\text{if $u$ adjacent to $v$},\\
0&\text{otherwise,}
\end{cases}
\]
where the notation $L_{uv}$ refers to the $(u,v)$ entry of the matrix $L$. It is common to consider what features of the graph $G$ may be determined via the eigenvalues of the Laplacian. Note that for any graph $G$, the row-sums of the Laplacian matrix are zero; hence it is immediate that $\lambda=0$ is an eigenvalue of $L$, with an eigenvector proportional to the all-ones vector.

In this article, we are interested in graphs whose Laplacian matrix can be diagonalized by a Hadamard matrix; that is, there exists a Hadamard matrix $H$ such that
\[\frac{1}{n} H^T L H = \Lambda,\]
or equivalently,
\[L = \frac{1}{n}H\Lambda H^T,\]
where $\Lambda$ is a diagonal matrix consisting of the eigenvalues of $L$, and noting that $H^{-1} = \frac{1}{n}H^T$. If this is the case, we refer to the graph as a Hadamard diagonalizable graph.
Clearly, this class of graphs corresponds to graphs for which there exists a full set of $\pm1$ orthogonal eigenvectors of the Laplacian matrix; i.e.\ there exists a collection of $n$ eigenvectors which correspond with the columns of a Hadamard matrix. 

Various properties of Hadamard diagonalizable graphs and a partial characterisation of Hadamard diagonalizable cographs were explored by Barik, Fallat, and Kirkland~\cite{HD1}. A special type of Hadamard matrices, called balancedly splittable Hadamard matrices, was introduced and studied by Kharaghani and Suda~\cite{HD3}; in particular, its connection to Hadamard digonalizable strongly regular graphs was made. Johnston, Kirkland, Plosker, Storey, and Zhang~\cite{HD2} showed that a graph is  diagonalizable by a Sylvester's matrix if and only if it is a cubelike graph (a Cayley graph over $\mathbb{Z}_2^d$). In the same paper,  Johnston et al. explored the use of Hadamard diagonalizable graphs in quantum information transfer, where a quantum spin network is represented by a graph and quantum information can transfer between spins. An important notion in quantum information transfer is perfect state transfer. Kay~\cite{PST} showed that a necessary condition for perfect state transfer between vertices $j$ and $k$ of a graph $G$, is that for a real orthogonal matrix $Q$ which diagonalizes $L(G)$, the corresponding entries in its $j$-th row and $k$-th row are either equal to or are the negative of each other. Hadamard diagonalizable graphs certainly satisfy this condition for any pair of vertices, and therefore these graphs are good candidates to admit perfect state transfer. 
A characterization of when a Hadamard diagonalizable graph admits perfect state transfer was given in terms of its eigenvalues and the normalized diagonalization Hadamard matrix in \cite{HD2}. 
Chan, Fallat, Kirkland, Lin, Nasserasr, and Plosker~\cite{CHD} studied complex Hadamard diagonalizable graphs (matrices $H$ with $H^*H = nI$, where the entries can be any complex number of modulus 1 rather than $\pm1$). Properties and constructions of such graphs were considered, as well as when such a graph admits interesting quantum information transfer phenomena. 

Most graphs are not Hadamard diagonalizable.  For example, they must have order $n=1$, $2$ or $4k$ (as Hadamard matrices only exist for these orders), but this is not sufficient. The following conditions are well-known.

\begin{prop}[\cite{HD1,HD2}]\label{prop:regular_even}
Let $G$ be a Hadamard diagonalizable graph. Then $G$ is regular; moreover, all eigenvalues must be even integers.
\end{prop}
We reproduce the proof of regularity here, and give an alternate proof that the eigenvalues must be even integers in Section~\ref{sec:computational_tool}.\\ 
\begin{proof}[Proof that the graph is regular]
Let $G$ be a Hadamard diagonalizable graph of order $n$, let $L$ be its Laplacian matrix, and let $H$ be a Hadamard matrix which diagonalizes $L$. The degrees of the vertices of the graph correspond to the diagonal entries of $L$.  Now let $h_k$ denote the $k$-th column of $H$, let $\lambda_1,\ldots,\lambda_n$ be the eigenvalues of $L$, and let $\Lambda = \text{diag}(\lambda_1, \cdots, \lambda_n)$. Then we have
\[
L=\frac1nH\Lambda H^T=\frac1n\sum_{k=1}^n\lambda_kh_kh_k^T.
\]
Since the diagonal entries of $h_k h_k^T$ are all equal to 1, the right hand side is a sum of matrices which all have constant diagonal. Hence $L$ has constant diagonal, and $G$ is regular.
\end{proof}

We will also make use of the following result.

\begin{prop}[\cite{HD1}]
A graph $G$ is Hadamard diagonalizable if and only if $G^c$ (the complement of $G$) is Hadamard diagonalizable.
\end{prop}

This follows immediately by noting that the eigenspaces for the Laplacian matrix of a graph and its complement are the same (although the eigenvalues are different).

Previous research into Hadamard diagonalizable graphs has characterized Hadamard diagonalizable graphs up through order $n=12$ (see \cite{HD1}), as well as all Hadamard diagonalizable graphs for the Sylvester construction for Hadamard matrices of order $2^k$ \cite{HD2}.  The goal of this current paper is to develop further tools to determine the Hadamard diagonalizable graphs of a given order, and to then list all Hadamard diagonalizable graphs up through order $n=36$.
We prove that for $n=8k+4$, there are only four Hadamard diagonalizable graphs of order $n$ in Section~\ref{sec:theory}, and we develop computational tools to search for all possible Hadamard diagonalizable graphs of small order in Section~\ref{sec:computational_tool}.  Information about the Hadamard diagonalizable graphs is given in Section~\ref{sec:small_graphs}.  Concluding comments will be given in Section~\ref{sec:conclusion}.

In Table~\ref{tab:data} we summarize the number of Hadamard diagonalizable graphs as well as the number of inequivalent Hadamard matrices of the indicated order.

\begin{table}
\centering
\caption{The order, number of non-equivalent Hadamard matrices (H.\ matrices), and the number of Hadamard diagonalizable graphs (H.\ graphs)}
\label{tab:data}
\begin{tabular}{|r|r|r|}\hline 
Order&H.\ matrices&H.\ graphs\\ \hline
$4$	    &	$1$          &  $4$ \\ \hline
$8$	    &	$1$          &  $10$ \\ \hline
$12$	&	$1$          &  $4$ \\ \hline
$16$	&	$5$          &  $50$ \\ \hline
$20$	&	$3$          &  $4$ \\ \hline
$24$	&	$60$         &  $26$ \\ \hline
$28$	&	$487$        &  $4$ \\ \hline
$32$	&	13,710,027   &  10,196 \\ \hline
$36$	&	(unknown)      &  $4$ \\ \hline
\end{tabular}
\end{table}

\section{Hadamard diagonalizable graphs  of order $n=8k+4$}\label{sec:theory}
We show that for order $n=8k+4$ there are at most four possible graphs which are Hadamard diagonalizable.   We start with the following graph characterization property.

\begin{lem}\label{lemma:four_path}
Suppose $G$ is a connected graph on $n$ vertices. Then $G$ is a complete graph or a complete bipartite graph if and only if the following condition holds: for any four distinct vertices $\{u,v,w, x\}\subseteq V(G)$ 
if $uv,vw,wx \in E(G)$ then $xu \in E(G)$.
\end{lem}
\begin{proof}
Assume that $G$ satisfies the stated condition for all distinct vertices $u,v,w,x\in V(G)$. 

If $G$ is acyclic, then $G$ contains no path on three edges, so $G$ is a star---that is, the complete bipartite graph $K_{1, n-1}$.

Now suppose that $G$ is not acyclic. Then the girth of $G$ (the length of a shortest cycle in $G$) is either $3$ or $4$.

Suppose the girth of $G$ is $3$, and let $U$ be a maximal clique.  Then $|U|\geq 3$.   Suppose for contradiction that $uv\in E(G)$ such that $u\in U$ and $v\not\in U$.  Since $v\not\in U$, there exists $x \in U$ such that $xv \not\in E(G)$.
Since $|U| \geq 3$, there exists $y \in U \setminus\{u,x\}$.
Then $xyuv$ is a path of length 3. By assumption, $xv$ are adjacent, which is a contradiction.
So it must be the case that no other vertices in $V(G)\setminus U$ are connected to a vertex in $U$; since $G$ is connected, we can conclude that $G$ is a complete graph. 

Now suppose the girth of $G$ is $4$, and let $U$ be a maximal induced complete bipartite subgraph of $G$, with bipartition $U = U_1\cup U_2$ such that $|U_1|,|U_2|\geq 2$.  	
By symmetry of $U_1$ and $U_2$, suppose for contradiction $uv\in E(G)$ where $u\in U_1$ and $v\not\in U$. 
If $v$ was adjacent to any vertex $z\in U_2$, there would be a triangle $uvz$, violating the girth condition. By the maximality of $U_2$, $v \not\in U_2$ because there exists $x \in U_1$ such that $xv \notin E(G)$.
Pick any $w \in U_2$. 
The path $vuwx$ of length 3 implies that $xv \in E(G)$, which is a contradiction.
So it must be the case that no other vertices in $G$ are connected to a vertex in $U$; and since $G$ is connected we can conclude that $G$ is a complete bipartite graph.  

The reverse implication holds by inspection.
\end{proof}

We can now use this characterization of graphs to establish the possible Hadamard diagonalizable graphs of order $n=8k+4$.

\begin{thm}
Let $G$ be a graph of order $n$. If $n=8k+4$ and $G$ is Hadamard diagonalizable, then $G\in \{K_n, K_{n/2,n/2}, nK_1, 2K_{n/2}\}$.  
\end{thm}
\begin{proof}
Suppose for the sake of contradiction that $n=8k+4$, $G$ is a Hadamard diagonalizable graph of order $n$, and $G\notin\{K_n, K_{n/2,n/2}, nK_1, 2K_{n/2}\}$. Let $L$ be the Laplacian matrix of $G$. Then there exists a diagonal matrix $\Lambda = \text{diag}(\lambda_1, \ldots, \lambda_n)$, where $\lambda_k$ is an eigenvalue of $L$ and $\lambda_k$ is an even integer, for all $k = 1, \ldots, n$ (see Proposition~\ref{prop:regular_even}) and 
\[
L = \frac1n H\Lambda H^T = \frac1n\sum_{k=1}^n\lambda_kh_kh_k^T,
\]
for some $n\times n$ Hadamard matrix $H$. 
For any $i,j\in \{1, \ldots, n\}$, 
\[
L_{ij}  =  \dfrac1n \sum_{k=1}^n \lambda_k (h_{k})_i(h_{k})_j  
\]
where the notation $(h_k)_i$ refers to the $i$-th entry of the vector $h_k$.

If $G$ is not connected, then the complement $G^c$ is connected. By Lemma~\ref{lemma:four_path}, then, we have that $G$ or $G^c$ contains a path of length $3$ whose endpoints are not adjacent (note that since $G$ must be regular, the only possible connected complete bipartite graph is $K_{n/2,n/2}$).  Without loss of generality, we assume $uvwx$ is a path of length $3$ in $G$. 
Since $L_{uv} = L_{vw} = L_{wx} = -1$ and $L_{ux} = 0$, we have 
\begin{align*}
-3n 
&= n(L_{uv} + L_{vw} + L_{wx} + L_{ux}) \\
&= \sum_{k=1}^n\lambda_k((h_{k})_u(h_{k})_v + (h_{k})_v(h_{k})_w + (h_{k})_w(h_{k})_x + (h_{k})_u(h_{k})_x)\\
&= \sum_{k=1}^n\lambda_k((h_{k})_u+(h_{k})_w)((h_{k})_v+(h_{k})_x).  
\end{align*}
Since each $\lambda_k$ is even and each $h_{ij}\in\{-1,1\}$, it follows that each term in the sum is divisible by $8$, meaning that $8$ divides the right hand side.  This implies that $n$ is a multiple of $8$.  But that contradicts the assumption that $n=8k+4$, concluding the proof. 
\end{proof}

The preceding result shows that if a graph is Hadamard diagonalizable of order $n=8k+4$ it must be one of the graphs mentioned.  We now must argue that all four of these graphs are realizable.

\begin{prop}\label{th:plus4}
If $n$ is even and there exists a Hadamard matrix of order $n$, then the graphs $K_n, K_{n/2,n/2}, nK_1,$ and $ 2K_{n/2}$ are Hadamard diagonalizable.
\end{prop}
\begin{proof}
Given that there exists a Hadamard matrix of order $n$ we may assume that there is a Hadamard matrix $H$ where $h_1$ is the all $1$s vector and $h_2$ is $1$ in entries $1,\ldots,n/2$ and $-1$ in entries $(n/2+1),\ldots,n$.  It suffices to show how to write $L$ as a linear combination of the projection matrices $h_kh_k^T$ for the graphs $K_n$ and $2K_{n/2}$ (since this will have the Laplacian with the correct eigenvalues).

For $G=K_n$ we have
\[
L=\sum_{k=1}^nh_kh_k^T-h_1h_1^T,
\]
since the sum becomes $nI$ and $h_1h_1^T$ is the all-ones matrix which we denote by $J$.

For $G=2K_{n/2}$ we have
\[
L=\frac12\sum_{k=1}^nh_kh_k^T-\frac12h_1h_1^T-\frac12h_2h_2^T,
\]
since the sum becomes $\frac{n}2I$ and the last two terms combine to give $-\big({J\atop O}\,{O\atop J}\big)$.
\end{proof}

\section{Finding all graphs diagonalizable by a given Hadamard matrix}\label{sec:computational_tool}

In this section, we describe a procedure to search for and produce Hadamard diagonalizable graphs. In particular, given a Hadamard matrix $H$, we give an algorithm by which all graphs which are diagonalized by $H$ are produced. We assume that $H$ is a normalized Hadamard matrix, since every graph which is Hadamard diagonalizable is also diagonalized by a normalized Hadamard matrix (see \cite[Lemma 4]{HD1}). We note that two  inequivalent normalized Hadamard matrices may produce the same graph via this procedure. 

\subsection{An algorithmic procedure}

Our Hadamard matrices $H$ will be assumed to be normalized Hadamard matrices which have the form
\[
H = \left[\begin{array}{c|c}1&1\cdots1\\\hline
1&\\
\vdots&\widehat H\\
1 \\
\end{array}\right]
\]
with $\widehat H$ a $\pm1$ matrix.  It is easily seen that every Hadamard matrix is equivalent to a matrix of this form by negating a combination of the rows and columns.    


\begin{prop}\label{prop:unique}
Let $G$ be a Hadamard diagonalizable graph with its Laplacian matrix $L$. Let $H$ be a normalized Hadamard matrix diagonalizing $L$. Let $\Lambda$ be the diagonal matrix with its diagonal entries $\lambda_1=0, \lambda_2, \dots, \lambda_n$, the eigenvalues of $L$ corresponding to the columns of $H$ as their associated eigenvectors.  Then the entries $L_{12},\ldots,L_{1n}$ uniquely determine  $\lambda_2,\ldots,\lambda_n$.
\end{prop}
\begin{proof}
Suppose that $G$ is a Hadamard diagonalizable graph, and let $H$ be a normalized Hadamard matrix such that
\[
L=\frac1nH\Lambda H^T=\frac1n\sum_{k=1}^n \lambda_k h_kh_k^T.\]
It follows that
\[L_{1j}  =  \dfrac1n \sum_{k=2}^n \lambda_k (h_{k})_j,
\]
where the notation $(h_k)_j$ refers to the $j^{th}$ entry of the vector $h_k$. Writing the above in matrix form we have
\[
\frac1n\left[\begin{array}{cccc}(h_2)_2&(h_3)_2&\cdots&(h_n)_2\\
(h_2)_3&(h_3)_3&\cdots&(h_n)_3\\
\vdots&\vdots&\ddots&\vdots\\
(h_2)_n&(h_3)_n&\cdots&(h_n)_n\end{array}\right]\left[\begin{array}{c}\lambda_2\\\lambda_3\\\vdots\\\lambda_n\end{array}\right]=
\frac1n\widehat{H}\left[\begin{array}{c}\lambda_2\\\lambda_3\\\vdots\\\lambda_n\end{array}\right] =\left[\begin{array}{c}L_{12}\\L_{13}\\\vdots\\L_{1n}\end{array}\right].
\]
The result now follows if we can prove that $\frac1n\widehat H$ is invertible, showing that we can solve for the $\lambda_i$ in terms of the off-diagonal entries in the first row.  In particular, we show that $\big(\frac1n\widehat H\big)^{-1}={\widehat H}^T-J$.

To prove this, we look at the rows of $\widehat{H}$. Note that if we append $1$s to the front, we have rows of $H$, and any two distinct rows in $H$ are perpendicular.  From this we can conclude that the dot product of two distinct rows in $\widehat{H}$ must be $-1$ (i.e.\ to compensate for the $1$s appended to the front); the dot product of a row in $\widehat{H}$ with the all $1$s vector must similarly be $-1$; finally, the dot product of a row with itself will be $n-1$.  

Multiplying $(\frac1n\widehat{H})(\widehat{H}^T-J)$ is equivalent to looking at dot products of rows in $\frac1n\widehat{H}$ and rows in $\widehat{H}-J$.  If the rows are the same, the result will be $\frac1n((n-1)-(-1))=1$; and if the rows are distinct the result will be $\frac1n((-1)-(-1))=0$.  In particular, the result is the identity matrix, establishing the inverse.
\end{proof}

The preceding can be used to give a new proof that all Laplacian eigenvalues of a Hadamard diagonalizable graph are even integers (see Proposition \ref{prop:regular_even}, originally proven in \cite{HD1}).

\begin{proof}[Proof that the eigenvalues are even integers]
We have
\[\left[\begin{array}{c}\lambda_2\\\lambda_3\\\vdots\\\lambda_n\end{array}\right] =(\widehat{H}^T-J)\left[\begin{array}{c}L_{12}\\L_{13}\\\vdots\\L_{1n}\end{array}\right].
\]
Since the entries in $\widehat{H}^T-J$ are in $\{0,-2\}$ while the entries $L_{12},\ldots,L_{1n}$ are in $\{0,-1\}$, the result of the multiplication will be a vector of integers which are even.
\end{proof}

Suppose we are given an $n\times n$ Hadamard matrix $H$ and wish to find all graphs which are Hadamard diagonalizable by $H$. Using Proposition~\ref{prop:unique}, we can narrow our search space down to size $2^{n-1}$ by looking at all possible $\{0, -1\}$ assignments to $L_{12},\ldots,L_{1n}$, and rewriting all of the off-diagonal entries of $L$ as linear combinations of $L_{12},\ldots,L_{1n}$.  This rewrite can be done because each entry is some linear combination of the $\lambda_2,\ldots,\lambda_n$, while the proof of Proposition~\ref{prop:unique} shows that each of the $\lambda_i$ is a linear combination  of $L_{12},\ldots,L_{1n}$. Then every assignment will produce a matrix via these linear combinations, though not every assignment will correspond to a graph, as there might be other entries $L_{ij}\notin\{-1,0\}$. The entry $L_{ij}=0$ if there is no edge between vertex $i$ and vertex $j$, and is equal to $-1$ if there is an edge. If $L_{ij}$ is some value other than $0$ or $-1$, we have not produced a Laplacian matrix. For any assignment of the values $\{0, -1\}$ to the `variables' $L_{12},\ldots,L_{1n}$, the goal will be to determine if the linear combinations that arise elsewhere in the matrix (off the diagonal) are all equal to either 0 or $-1$; if so, that assignment produces a graph. It is also possible to construct the same graph multiple ways (i.e.\ the same up to relabeling of the vertices).
Note that many distinct entries of the Laplacian may be expressed using the same linear combination of the variables $L_{12},\ldots,L_{1n}$.

To illustrate this, we carry this procedure out for the Hadamard matrix \texttt{had.16.1} from Sloane \cite{Sloane}, see Table~\ref{had.16.1}, to produce an auxiliary matrix determining the linear combinations.  For the $120$ entries above the diagonal of a possible Laplacian matrix (by symmetry the entries below the diagonal will be equal) there were $27$ distinct linear combinations produced. The auxiliary matrix is given in Table~\ref{tab:16example} where the $(i,j)^{th}$ entry corresponds to the coefficient of $L_{1j}$ in the $i^{th}$ linear combination.  For notational convenience, we have labeled the sixteen rows and columns of $L$ using hexadecimal symbols $\{0, 1, \cdots, 9, A, B, \cdots, F\}$ to more easily indicate which entries of $L$ correspond to the $i^{th}$ linear combination in the accompanying table.

\begin{table}\label{had.16.1}
    \caption{Hadamard matrix \texttt{had.16.1}}
    \label{tab:my_label}
    
\[
\small
\left[\begin{array}{rrrrrrrrrrrrrrrr}
  1 &  1 &  1 &  1 &  1 &  1 &  1 &  1 &  1 &  1 &  1 &  1 &  1 &  1 &  1 &  1 \\
  1 & -1 &  1 & -1 &  1 & -1 &  1 & -1 &  1 & -1 &  1 & -1 &  1 & -1 &  1 & -1 \\
  1 &  1 & -1 & -1 &  1 &  1 & -1 & -1 &  1 &  1 & -1 & -1 &  1 &  1 & -1 & -1 \\
  1 & -1 & -1 &  1 &  1 & -1 & -1 &  1 &  1 & -1 & -1 &  1 &  1 & -1 & -1 &  1 \\
  1 &  1 &  1 &  1 & -1 & -1 & -1 & -1 &  1 &  1 &  1 &  1 & -1 & -1 & -1 & -1 \\
  1 & -1 &  1 & -1 & -1 &  1 & -1 &  1 &  1 & -1 &  1 & -1 & -1 &  1 & -1 &  1 \\
  1 &  1 & -1 & -1 & -1 & -1 &  1 &  1 &  1 &  1 & -1 & -1 & -1 & -1 &  1 &  1 \\
  1 & -1 & -1 &  1 & -1 &  1 &  1 & -1 &  1 & -1 & -1 &  1 & -1 &  1 &  1 & -1 \\
  1 &  1 &  1 &  1 &  1 &  1 &  1 &  1 & -1 & -1 & -1 & -1 & -1 & -1 & -1 & -1 \\
  1 & -1 &  1 & -1 &  1 & -1 & -1 &  1 & -1 &  1 & -1 &  1 & -1 &  1 &  1 & -1 \\
  1 &  1 & -1 & -1 &  1 &  1 & -1 & -1 & -1 & -1 &  1 &  1 & -1 & -1 &  1 &  1 \\
  1 & -1 & -1 &  1 &  1 & -1 &  1 & -1 & -1 &  1 &  1 & -1 & -1 &  1 & -1 &  1 \\
  1 &  1 &  1 &  1 & -1 & -1 & -1 & -1 & -1 & -1 & -1 & -1 &  1 &  1 &  1 &  1 \\
  1 & -1 &  1 & -1 & -1 &  1 &  1 & -1 & -1 &  1 & -1 &  1 &  1 & -1 & -1 &  1 \\
  1 &  1 & -1 & -1 & -1 & -1 &  1 &  1 & -1 & -1 &  1 &  1 &  1 &  1 & -1 & -1 \\
  1 & -1 & -1 &  1 & -1 &  1 & -1 &  1 & -1 &  1 &  1 & -1 &  1 & -1 &  1 & -1 \\
\end{array}
\right]
\]
\end{table}

\begin{table}
\centering
\caption{The auxiliary matrix for the Hadamard matrix \texttt{had.16.1} where each row corresponds with a distinct linear combination appearing in $L$ in terms of the off-diagonal entries in the first row.  In the accompanying table, we indicate for each row which entries $L_{ij}$ correspond to this linear combination, with $i,j$ in hexadecimal.}
\label{tab:16example}
\[
\left[\begin{array}{rrrrrrrrrrrrrrr}
1 & 0 & 0 & 0 & 0 & 0 & 0 & \phantom{-}0 & \phantom{-}0 & 0 & 0 & 0 & 0 & 0 & 0 \\
0 & 1 & 0 & 0 & 0 & 0 & 0 & 0 & 0 & 0 & 0 & 0 & 0 & 0 & 0 \\
0 & 0 & 1 & 0 & 0 & 0 & 0 & 0 & 0 & 0 & 0 & 0 & 0 & 0 & 0 \\
0 & 0 & 0 & 1 & 0 & 0 & 0 & 0 & 0 & 0 & 0 & 0 & 0 & 0 & 0 \\
0 & 0 & 0 & 0 & 1 & 0 & 0 & 0 & 0 & 0 & 0 & 0 & 0 & 0 & 0 \\
0 & 0 & 0 & 0 & 0 & 1 & 0 & 0 & 0 & 0 & 0 & 0 & 0 & 0 & 0 \\
0 & 0 & 0 & 0 & 0 & 0 & 1 & 0 & 0 & 0 & 0 & 0 & 0 & 0 & 0 \\
0 & 0 & 0 & 0 & 0 & 0 & 0 & 1 & 0 & 0 & 0 & 0 & 0 & 0 & 0 \\
0 & 0 & 0 & 0 & 0 & 0 & 0 & 0 & 1 & 0 & 0 & 0 & 0 & 0 & 0 \\
0 & 0 & 0 & 0 & 0 & 0 & 0 & 0 & 0 & 1 & 0 & 0 & 0 & 0 & 0 \\
0 & 0 & 0 & 0 & 0 & 0 & 0 & 0 & 0 & 0 & 1 & 0 & 0 & 0 & 0 \\
0 & 0 & 0 & 0 & 0 & 0 & 0 & 0 & 0 & 0 & 0 & 1 & 0 & 0 & 0 \\
0 & 0 & 0 & 0 & 0 & 0 & 0 & 0 & 0 & 0 & 0 & 0 & 1 & 0 & 0 \\
0 & 0 & 0 & 0 & 0 & 0 & 0 & 0 & 0 & 0 & 0 & 0 & 0 & 1 & 0 \\
0 & 0 & 0 & 0 & 0 & 0 & 0 & 0 & 0 & 0 & 0 & 0 & 0 & 0 & 1 \\
0 & 0 & 0 & 0 & 0 & \frac{1}{2} & \frac{1}{2} & 0 & 0 & 0 & 0 & 0 & 0 & \frac{1}{2} & -\frac{1}{2} \\
0 & 0 & 0 & 0 & 0 & \frac{1}{2} & \frac{1}{2} & 0 & 0 & 0 & 0 & 0 & 0 & -\frac{1}{2} & \frac{1}{2} \\
0 & 0 & 0 & \frac{1}{2} & \frac{1}{2} & 0 & 0 & 0 & 0 & 0 & 0 & \frac{1}{2} & -\frac{1}{2} & 0 & 0 \\
0 & 0 & 0 & \frac{1}{2} & \frac{1}{2} & 0 & 0 & 0 & 0 & 0 & 0 & -\frac{1}{2} & \frac{1}{2} & 0 & 0 \\
0 & 0 & 0 & 0 & 0 & \frac{1}{2} & -\frac{1}{2} & 0 & 0 & 0 & 0 & 0 & 0 & \frac{1}{2} & \frac{1}{2} \\
0 & 0 & 0 & 0 & 0 & -\frac{1}{2} & \frac{1}{2} & 0 & 0 & 0 & 0 & 0 & 0 & \frac{1}{2} & \frac{1}{2} \\
0 & 0 & 0 & \frac{1}{2} & -\frac{1}{2} & 0 & 0 & 0 & 0 & 0 & 0 & \frac{1}{2} & \frac{1}{2} & 0 & 0 \\
0 & 0 & 0 & -\frac{1}{2} & \frac{1}{2} & 0 & 0 & 0 & 0 & 0 & 0 & \frac{1}{2} & \frac{1}{2} & 0 & 0 \\
0 & \frac{1}{2} & \frac{1}{2} & 0 & 0 & 0 & 0 & 0 & 0 & \frac{1}{2} & -\frac{1}{2} & 0 & 0 & 0 & 0 \\
0 & \frac{1}{2} & \frac{1}{2} & 0 & 0 & 0 & 0 & 0 & 0 & -\frac{1}{2} & \frac{1}{2} & 0 & 0 & 0 & 0 \\
0 & \frac{1}{2} & -\frac{1}{2} & 0 & 0 & 0 & 0 & 0 & 0 & \frac{1}{2} & \frac{1}{2} & 0 & 0 & 0 & 0 \\
0 & -\frac{1}{2} & \frac{1}{2} & 0 & 0 & 0 & 0 & 0 & 0 & \frac{1}{2} & \frac{1}{2} & 0 & 0 & 0 & 0
\end{array}\right]
\]
\begin{tabular}{|c|c|}\hline
Row&Entries\\ \hline
$1$&$01,23,45,67,89,AB,CD,EF$\\ \hline
$2$&$02,13,8A,9B$\\ \hline
$3$&$03,12,8B,9A$\\ \hline
$4$&$04,15,8C,9D$\\ \hline
$5$&$05,14,8D,9C$\\ \hline
$6$&$06,17,8E,9F$\\ \hline
$7$&$07,16,8F,9E$\\ \hline
$8$&$08,19,2A,3B,4C,5D,6E,7F$\\  \hline
$9$&$09,18,2B,3A,4D,5C,6F,7E$\\  \hline
\end{tabular}\hfil
\begin{tabular}{|c|c|} \hline
Row&Entries\\ \hline
$10$& $0A,1B,28,39$\\ \hline
$11$& $0B,1A,29,38$\\ \hline
$12$& $0C,1D,48,59$\\ \hline
$13$& $0D,1C,49,58$\\ \hline
$14$& $0E,1F,68,79$\\ \hline
$15$& $0F,1E,69,78$\\ \hline
$16$& $24,35,AC,BD$\\ \hline
$17$& $25,34,AD,BC$\\ \hline
$18$& $26,37,AE,BF$\\ \hline
\end{tabular}\hfil
\begin{tabular}{|c|c|} \hline
Row&Entries\\ \hline
$19$& $27,36,AF,BE$\\ \hline
$20$& $2C,3D,4A,5B$\\ \hline
$21$& $2D,3C,4B,5A$\\ \hline
$22$& $2E,3F,6A,7B$\\ \hline
$23$& $2F,3E,6B,7A$\\ \hline
$24$& $46,57,CE,DF$\\ \hline
$25$& $47,56,CF,DE$\\ \hline
$26$& $4E,5F,6C,7D$\\ \hline
$27$& $4F,5E,6D,7C$\\ \hline
\end{tabular}
\end{table}

Looking at the auxiliary matrix in Table~\ref{tab:16example}, the identity matrix induced on the first $15$ rows is a reflection that the linear combination for an entry from the first row is trivial.

Let us view an assignment of $L_{12},\ldots,L_{1n}$ to values of $0$ or $-1$ as selecting some subset of the columns (so if $L_{1j}$ is $-1$ take the $j^{th}$ column; if it is $0$ do not take the column).  Then this will produce a Laplacian matrix for a graph \emph{if and only if} the sum of the columns produce a vector with entries in $\{0, 1\}$, since this is only the case when the off-diagonal entries will be $0$ and $-1$.

To run through all possible assignments of $L_{12},\ldots,L_{1n}$, we need to consider all subsets of columns of the auxiliary matrix. However, to reduce the search space, we proceed via a tree-like exploration of the space, where at each step we decide to either add or not add a particular column. After we add a new column to our subset, we then do the following check: if for each entry there is a possibility that the sum of some combination of the remaining columns can result in the value being $0$ or $1$, proceed; if not then we `prune the tree' and don't explore any further on that branch. 

For example, if we take the columns $2$, $3$, and $11$ in the matrix in Table~\ref{tab:16example}, then the last entry in the sum of these column vectors will be $1/2$. The last entries in the remaining columns which we could add to our subset are all equal to $0$, so no matter which combination of columns $12,13,14,15$ we take we can never change that value from $1/2$, and so there is no need to explore that part of the space.  To get the most out of this, it is useful to first pre-sort the columns so that such conflicts will arise early.

If we get down to a leaf in the tree and the resulting combination of columns is a $0$-$1$ vector, then we have found a Hadamard diagonalizable graph.  To produce the graph we find where the $1$s are located and the corresponding Laplacian entries to which they correspond.  These corresponding Laplacian entries represent the edges in the graph.  As an example if we take the sum of the first three columns in Table~\ref{tab:16example} then this will produce a $1$ in rows $1,2,3,24,$ and $25$ of the resulting vector.  So this will be the graph on the vertex set with vertices $\{0,1,\ldots,F\}$ and with edges
\[
\underbrace{01,23,45,67,89,AB,CD,EF}_{\text{row $1$}},
\underbrace{02,13,8A,9B}_{\text{row $2$}},
\underbrace{03,12,8B,9A}_{\text{row $3$}},
\underbrace{46,57,CE,DF}_{\text{row $24$}},
\underbrace{47,56,CF,DE}_{\text{row $25$}}
\]
which becomes the graph $4K_4$ (cliques on the vertices $0,1,2,3$; and $4,5,6,7$; and $8,9,A,B$; and $C,D,E,F$).  As graphs are found they are tested to see whether they have been seen before, and we only keep those graphs which have not been seen before; this can be done, for example, by using canonical labeling methods.

The procedure outlined here was implemented in both \texttt{SageMath} and \texttt{C++} with all computations done using integer variables.  The only external call needed is to determine which graphs are discovered up to isomorphism. The program can be downloaded at \oururl.

\subsection{Equivalency of Hadamard matrices}

We know that if $G$ is Hadamard diagonalizable, then it is Hadamard diagonalizable by some normalized Hadamard matrix. However, given a normalized Hadamard matrix $H$, there are many other normalized Hadamard matrices which are equivalent to $H$ (i.e.\ obtained from $H$ via some sequence of operations from negating rows, negating columns, permuting rows or permuting columns). It is not the case that if $G$ is diagonalized by $H$ that it is also diagonalizable by any $H'$ equivalent to $H$. For example, consider the standard normalized Hadamard matrix of order 4:
\[H = \left[\begin{array}{rrrr} 1 & 1 & 1 & 1 \\ 1 & -1 & 1 & -1 \\ 1 & 1 & -1 & -1 \\ 1 & -1 & -1 & 1 \end{array}\right].\]
This Hadamard diagonalizes the complete graph $K_4$. However, an equivalent Hadamard matrix is obtained by negating the second row:
\[H' = \left[\begin{array}{rrrr} 1 & 1 & 1 & 1\\ -1 & 1 & -1 & 1 \\ 1 & 1 & -1 & -1 \\ 1 & -1 & -1 & 1 \end{array}\right].\]
Note that every Laplacian matrix has 0 as an eigenvalue, with an eigenvector proportional to the all-ones vector. Since the columns of a Hadamard matrix that diagonalize a graph represent the eigenvectors of that graph's Laplacian, and this matrix $H'$ has no constant column, it is clear that $H'$ does not diagonalize any connected graph.

We suggest, however, that it may be possible to show the following:
\begin{conj}\label{permcols}
If $H_1$ and $H_2$ are equivalent normalized Hadamard matrices, then $G$ is Hadamard diagonalizable by $H_1$ if and only if $G$ (up to some relabeling) is Hadamard diagonalizable by $H_2$.
\end{conj}
If true, this conjecture would significantly shorten the computational time.

We note, in addition, that of the four operations by which an equivalent Hadamard matrix is produced, three of them preserve the graphs diagonalized by that Hadamard matrix. Permuting columns of $H$ corresponds to a permutation of the eigenvectors of $L$; permuting rows corresponds simply to a relabelling of the vertices of the graph. Negating columns does not change the eigenspaces, it simply scales our representative eigenvectors of $L$. 


\section{Hadamard diagonalizable graphs of small order}\label{sec:small_graphs}

In this section, we present complete lists of all Hadamard diagonalizable graphs for all orders $n \in \{1, 2, 4, 8, 12, 16, 20, 24, 28, 32, 36\},$ obtained using the theoretical tools of Section~\ref{sec:theory} and the computational tools of Section~\ref{sec:computational_tool}. Hadamard diagonalizable graphs are characterized in the literature up to order 12 in Barik, Fallat, and Kirkland~\cite{HD1}, but we present these as well for completeness. The most significant contributions here on the computational side are for orders $n=16, 24, 32$.



We will use the following notations:
\begin{itemize}
    \item $G^c$ is the graph complement of $G$.  
    \item $G+H$ is the disjoint union of the graphs $G$ and $H$. In particular, we denote by $kG$ the graph consisting of $k$ disjoint copies of $G$.
    \item $G\square H$ is the Cartesian product of $G$ and $H$. 
          That is $V(G\square H) = V(G) \times V(H)$ and $(u,x),(v,y) \in V(G\square H)$ are adjacent if and only if  $u=v$ and $xy \in E(H)$ or  $x=y$ and $uv \in E(G)$.
    \item $G\vee H$ is the join of $G$ and $H$. The join is obtained from $G+H$ by adding all edges $uv$, where $u \in V(G)$ and $v \in V(H)$.
    \item $G\wr H$ is the \emph{lexicographic product} of $G$ with $H$; that is, the graph formed by replacing each vertex of $G$ with a copy of $H$, and adding all possible edges between the vertices in the copies of $H$ corresponding to adjacent vertices in $G$ (a form of a blow-up). This product is sometimes referred to as \emph{graph composition}. It has also occasionally appeared in the past as a \emph{wreath product}, due to its connection with wreath products in group theory (see \cite{We}).
    \item $H_{n,n}$ is the graph $K_{n,n}$ minus a perfect matching (note that $H_{4,4}=Q_3$ is the cube graph on eight vertices); \item $CP_{2n}$ is the cocktail party graph on $2n$ vertices formed by taking the complete graph and removing a perfect matching, so $CP_{2n}=K_{2,2,\ldots,2}$.
    \item Let $\mathcal{G}$ be a finite group, and let $S$ be some subset of the elements of $\mathcal{G}$. Then $\mathcal{G}(S)$ is a \emph{Cayley graph} with vertices representing elements of the group $\mathcal{G}$, and an edge between $u$ and $v$ whenever $u-v \in S$.
\end{itemize}

There are many ways to write some of the graphs in what follows, and in the interest of future theoretical research in this area (that is, a pursuit of theoretical characterizations of Hadamard diagonalizable graphs), we will often give several isomorphic representations of the same graph. In particular, many can be written as a Cartesian or lexicographic product of two graphs. This is pursued in earnest for the graphs of order $24$, as we conjecture (based on these and the preliminary data available for orders $40$ and $56$) that in a manner similar to orders $8k+4, k\ge 0$, there are at most 26 distinct graphs which are Hadamard diagonalizable of order $16k+8$ (see more discussions Conjecture~\ref{conj26}). 

It is shown in \cite{HD1} that $G+G\cong 2K_1\square G$ and $G\vee G\cong K_2\wr G$ are Hadamard diagonalizable graphs if $G$ is Hadamard diagonalizable, and that the Cartesian product of two Hadamard diagonalizable graphs is also Hadamard diagonalizable. We now show that the lexicographic product of two Hadamard diagonalizable graphs is also Hadamard diagonalizable, giving further methods to construct Hadamard diagonalizable graphs.

\begin{lem}\label{th:wr}
Let $G_1$ and $G_2$ be Hadamard diagonalizable. Then $G_1\wr G_2$ is Hadamard diagonalizable.
\end{lem}
\begin{proof}
Assume that the graph $G_1$ on $m$ vertices is diagonalized by a normalized Hadamard matrix $H_1$, and the graph $G_2$ on $n$ vertices is diagonalized by a normalized Hadamard matrix $H_2$. Since Hadamard diagonalizable graphs are regular, a Hadamard matrix $H$ diagonalizes the Laplacian of a graph $L(G)$ if and only if it also diagonalizes the adjacency matrix, which we denote $A(G)$. 
We show that the adjacency matrix of the lexicographic product of $G_1$ and $G_2$, $A(G_1\wr G_2)$, is diagonalizable by the Hadamard matrix $H_1\otimes H_2$, where $\otimes$ denotes the Kronecker product (or tensor product) of two matrices.

Assume that 
\[H_1^{-1}A(G_1)H_1 = \Lambda_1, \quad \text{and} \quad H_2^{-1}A(G_2)H_2 = \Lambda_2. \]
The adjacency matrix of the lexicographic product can be written
\[A(G_1\wr G_2)=I_m\otimes A(G_2)+A(G_1)\otimes J_n,\]
where $I_m$ is the identity of order $m$ and $J_n$ is the $n \times n$ matrix of all ones. 

For any normalized Hadamard matrix $H$ of size $n$, $H^{-1}J_nH=ne_1e_1^T=nE_{1,1}$, where $e_1$ is the vector with a 1 in the first position and zeros elsewhere, and $E_{1, 1}$ is a matrix of all-zeros except a 1 in the $(1, 1)$ position. Using this and the fact that the Kronecker product is bilinear and satisfies the properties, $(AB)\otimes(CD) = (A\otimes C)(B\otimes D)$ and $(A\otimes B)^{-1}=A^{-1}\otimes B^{-1}$, we have
\begin{align*}
(H_1\otimes H_2)^{-1}A(G_1\wr G_2)(H_1\otimes H_2)
&=(H_1\otimes H_2)^{-1}(I_m\otimes A(G_2)+A(G_1)\otimes J_n)(H_1\otimes H_2)\\
&=(H_1^{-1}I_mH_1)\otimes (H_2^{-1}A(G_2)H_2)\\ & \qquad \qquad \qquad+(H_1^{-1}A(G_1)H_1)\otimes (H_2^{-1}J_nH_2)\\
&=(I_m\otimes \Lambda_2)+(n\Lambda_1\otimes E_{1,1}),
\end{align*}
which is a diagonal matrix. 
\end{proof}





Thus the set of all Hadamard diagonalizable graphs are closed under the Cartesian and lexicographic product, as the Cartesian product or lexicographic product of two Hadamard diagonalizable graphs is Hadamard diagonalizable. 

We also note that for any given Hadamard diagonalizable graphs $G_1$ and $G_2$ with corresponding diagonalizing Hadamard matrices $H_1$ and $H_2$, respectively, it is interesting to see that the Hadamard matrix $H_1\otimes H_2$ diagonalizes both $G_1\square G_2$ and  $G_1\wr G_2$. As two graphs $G_1\square G_2$ and  $G_1\wr G_2$ are nonisomorphic, in general, they may admit different Hadamard matrices as their diagonalizing matrices as well. See for example, $K_2\square K_{6,6}$ and $K_2\wr K_{6,6}$ in Table \ref{tab:had24} below.

\subsection{Order $1$}
The graph $K_1$ is Hadamard diagonalizable.

\subsection{Order $2$}
Both graphs $K_2$ and $2K_1$ are Hadamard diagonalizable by the unique normalized Hadamard matrix $\left[\begin{array}{rr} 1 & 1 \\ 1 & -1\end{array}\right]$.

\subsection{Orders $4$, $12$, $20$, $28$ and $36$}
From the results of Section~\ref{sec:theory}, the only Hadamard diagonalizable graphs are $K_{n}$, $K_{n/2, n/2}$, $2K_{n/2}$, and $n K_1$ for $n \in \{4,12,20,28,36\}$.

\subsection{Order $8$}
There is a unique (up to equivalence) normalized Hadamard matrix of order $8$, and it is the Sylvester construction. The graphs diagonalizable by a Sylvester Hadamard matrix have been characterized in \cite{HD2}, and for order $8$ consist of all Cayley graphs for $\mathbb{Z}_2^3$. We can also represent them as follows:
\begin{table}[h]
\centering
\caption{Hadamard diagonalizable graphs of order $8$}
\label{tab:had8}
\begin{tabular}{|l|l|}\hline
 Graph & Graph complement \\ \hline \hline
 $K_2\wr K_4 \cong K_8$ &  $2K_1\square 4K_1 \cong 2(4K_1)$ \\\hline 
 $K_2\wr K_{2,2}$ &   $K_2\square 4K_1 \cong 2K_1\square 2K_2 \cong 2(2K_2)$ \\\hline
 $K_2\wr 2K_2$  &  $K_2\square 2K_2\ \cong \  2K_1\square K_{2,2} \cong 2K_{2,2}$ \\\hline
 $K_2\wr 4K_1\ \cong \  K_{4,4}$ &  $2K_1\square K_4 \cong 2K_4$ \\\hline 
 $K_2\square K_4$ & $K_{2}\square K_{2,2}  \cong   (K_2)^{\square 3}\cong Q_3$ \\ \hline
\end{tabular}
\end{table}
Note that all of them can be expressed in terms of products of the Hadamard diagonalizable graphs of orders $2$ and $4$.


\subsection{Order $16$}
There are five non-equivalent normalized Hadamard matrices of order $16$.  We will follow Sloane \cite{Sloane} and denote them by \texttt{had.16.$j$} for $j\in\{0,1,2,3,4\}$; note that \texttt{had.16.0} is the Sylvester construction. We use the computational tools of Section~\ref{sec:computational_tool} to produce all graphs diagonalized by one (or more) of these Hadamard matrices.

There are a total of $50$ Hadamard diagonalizable graphs on $16$ vertices, all of which are Cayley graphs. The graphs are given in Table~\ref{tab:had16}. Many of these graphs, not all, can be identified as the products of smaller Hadamard diagonalizable graphs, products of order 2 and order 8, products of order 4 and identical or another graph of order 4, etc. Graphs come in pairs (namely the graph and its complement) and we sometimes only present one of the graphs (to get the other take the complement, for which a simple product notation is not available). We also provide the Cayley expression for some graphs, including the strongly regular graphs, Shrikhande graph, and its cospectral mate, the $(2, 4)$-Hamming graph $K_4\square K_4$.

The column indicating `Family' presents information regarding the graphs which are diagonalized by the same Hadamard matrices---that is, if $G$ and $H$ are in the same family, then any Hadamard matrix diagonalizing the Laplacian of $G$ will also diagonalize the Laplacian of $H$. We now list which graphs are associated with which Hadamard matrices as follows:
\begin{itemize}
\item $46$ graphs for \texttt{had.16.0} are from families A, B, C, D
\item $50$ graphs for \texttt{had.16.1} are from families A, B, C, D, E
\item $48$ graphs for \texttt{had.16.2} are from families A, B, C, E
\item $24$ graphs for \texttt{had.16.3} are from families A, B
\item $10$ graphs for \texttt{had.16.4} are from family A
\end{itemize}

\begin{table}[h]
\centering
\caption{Hadamard diagonalizable graphs of order $16$}
\label{tab:had16}
\begin{tabular}{|c|l|l|}\hline
Family&Graph \hspace{100pt} & Graph complement \\ \hline \hline
A& $K_{16}$ & $16K_1$ \\ \hline
A& $K_{8,8}$ & $2K_8$ \\ \hline
A& $2K_{4,4}$ & $K_2\wr 2K_4$ \\ \hline
A& $4K_4$ & $K_{4,4,4,4}$ \\ \hline
A& $(K_2\square K_4)\wr (2K_1)$ & $H_{4,4}\wr K_2$ \\ \hline \hline
B& $8K_2$ & $K_8\wr 2K_1$ \\ \hline
B& $4K_{2, 2}$ & $K_4\wr(2K_2)$ \\ \hline
B& $K_2\wr 4K_2$ & $2K_1\wr (K_4\wr 2K_1)$\\ \hline
B& $K_2\wr 2K_{2, 2}$ & $2K_1\wr (K_2\wr 2K_2)$\\ \hline
B& $(K_2\square K_4)\wr K_2$ & $H_{4,4}\wr K_2^c$\\ \hline
B& \multicolumn{2}{l|}{$K_2\square (K_4\wr 2K_1)$}\\ \hline
B& \multicolumn{2}{l|}{$K_2\square K_8$}\\ \hline \hline
C& \multicolumn{2}{l|}{$K_{4,4}\square K_2$} \\ \hline
C& \multicolumn{2}{l|}{$2(K_4\square K_2)$}\\ \hline
C& \multicolumn{2}{l|}{$2(K_{2, 2}\wr K_2)$}\\ \hline
C& \multicolumn{2}{l|}{$K_{2, 2}\square K_{2, 2}$}\\ \hline
C& \multicolumn{2}{l|}{$K_{2, 2}\wr K_4$} \\ \hline
C& \multicolumn{2}{l|}{$2H_{4,4}$}\\ \hline
C& \multicolumn{2}{l|}{$\mathbb{Z}_2^4(\{(0, 0, 0, 1), (0, 0, 1, 0), (0, 1, 0, 0), (1, 0, 0, 0), (1, 1, 1, 1)\})$}\\ \hline
C& \multicolumn{2}{l|}{$\mathbb{Z}_2^4(\{(0, 0, 0, 1), (0, 0, 1, 0), (0, 1, 0, 0), (1, 0, 0, 0), (0, 0, 1, 1), (1, 1, 0, 1\})$}\\ \hline
C& \multicolumn{2}{l|}{$\mathbb{Z}_2^4(\{(0, 0, 0, 1), (0, 0, 1, 0), (0, 1, 0, 0), (1, 0, 0, 0), (0, 0, 1, 1), (1, 1, 0, 0)\})$ ((2,4)-Hamming)}\\ \hline
C& \multicolumn{2}{l|}{$\mathbb{Z}_2^4(\{(0, 0, 0, 1), (0, 0, 1, 0), (0, 1, 0, 0), (1, 0, 0, 0), (0, 0, 1, 1), (0, 1, 0, 1), (1, 1, 1, 0)\})$}\\ \hline \hline
D& \multicolumn{2}{l|}{$\mathbb{Z}_2^4(\{(0,0,1,1),(0,1,0,0),(0,1,0,1),(0,1,1,0),(0,1,1,1),(1,0,0,0),(1,1,0,0)\})$}\\ \hline \hline
E& \multicolumn{2}{l|}{$\mathbb{Z}_4^2(\{(0,1),(0,-1),(1,0),(-1,0),(1,1),(-1,-1)\})$ (Shrikhande)}\\ \hline
E& \multicolumn{2}{l|}{$\mathbb{Z}_4^2(\{(0,1),(0,-1),(1,0),(-1,0),(1,1),(-1,-1),(2,2)\})$}\\ \hline
\end{tabular}
\end{table}

\subsection{Order $24$}
There are 60 non-equivalent normalized Hadamard matrices of order $24$.  We will follow Sloane \cite{Sloane} and denote them by \texttt{had.24.$j$} for $j\in\{1,2,\ldots,60\}$.

There are a total of $26$ Hadamard diagonalizable graphs on 24 vertices, all of which are Cayley graphs. The graphs are given in Table~\ref{tab:had24}. Graphs are presented in complementary pairs, with multiple representations for each. The graph $12K_2$ as a subgraph of $G$ will be denoted by $M$ (a matching) and the graph $G-M$ refers to the graph obtained by the removal of the edges of a matching from $G$. Again, we indicate equivalence classes in the column `Family', where graphs from the same family are diagonalized by the same Hadamard matrices. Note that we use the notation $G\times H$ when $G\wr H$ is isomorphic to $G\square H$.


\begin{table}[h]
\centering
\caption{Hadamard diagonalizable graphs of order $24$}
 \label{tab:had24}
\begin{tabular}{|c|l|lr|}\hline
Family & Graph &  Graph complement & \\  \hline \hline 
A & $K_2\wr K_{12} \cong K_{24}$ & $2K_1\times 12K_1 \cong 24K_1$ &\\ \hline 
A & $K_2\wr 12K_1\cong K_{12,12}$ & $2K_1\times K_{12}\ \cong\ 2K_{12}$ &\\ \hline \hline
B & $K_2\wr K_{6,6}\ \cong\ K_{6,6,6,6}$  & $2K_1\times 2K_6 \cong 4K_6$ &\\ \hline 
B & $K_2\wr 2K_6\ \cong \  K_{2, 2}\wr K_6$ & $2K_1\times K_{6,6}\ \cong \  2K_{6,6}$  &\\  \hline
B & $(K_2)^{\square 3}\wr K_3\ \cong \  Q_3\wr K_3$ & $(K_4\square K_2)\wr 3K_1$ & $\ast$ \\ \hline \hline 
C & $K_2\square 12K_1 \cong 12K_2$ & $K_2\wr K_{12}-M \cong CP_{24} $&\\ \hline 
C &  $K_{2}\square K_{12}$ & $K_2\wr 12K_1-M\cong  H_{12,12}$ &\\ \hline 
C &  $K_{6,6}\wr K_2$ &  $2K_6\wr 2K_1\ \cong\ 2K_{12}-M$ & \\ \hline \hline
D & $K_2\square K_{6,6}$ & $K_2\wr 2K_6 -M\cong K_{2, 2}\wr K_6-M$ & \\ \hline 
D & $K_2\square 2K_6$ & $K_2\wr K_{6,6}-M$ & \\ \hline
D &  $K_2\wr (K_6\square K_2)$ & $2(H_{6, 6})\cong 2K_{6,6}-M$ & $\ast$\\ \hline 
D & $(K_4\square K_2)\wr 3K_1 - M$  &  $Q_3\wr K_3 + M$ & $\ast$\\ \hline 
D &  $(K_4\square K_2)\wr 3K_1 + M$ &   $Q_3\wr K_3 - M$  & $\ast$ \\ \hline
\end{tabular}
\end{table}
   

We now list which graphs are associated with which Hadamard matrices as follows:
\begin{itemize}
\item $26$ graphs for \texttt{had.24.$j$} for $1\le j\le 7$ are from families A, B, C, D
\item $10$ graphs for \texttt{had.24.8} are graphs from families A, C
\item $10$ graphs for \texttt{had.24.$j$} for $9\le j\le 59$ are from families A, B
\item $4$ graphs for \texttt{had.24.60} are from family A
\end{itemize}

One interesting thing to note is that the graph $2K_1\wr (K_6\wr 2K_1)$ (i.e. two disjoint copies of the cocktail party graph $CP_{12}$) is Hadamard diagonalizable by a Hadamard matrix of order $24$, but the cocktail party graph $CP_{12}\cong K_6\wr 2K_1$ is \emph{not} diagonalized by the unique Hadamard matrix of order $12$. In fact, those graphs in the table that are marked by an asterisk can be expressed as products of smaller graphs but some of their factors are not necessarily Hadamard diagonalizable ones.
However, all remaining graphs are coming as the products of Hadamard diagonalizable graphs of smaller orders as expressed in the product notation. 
This guarantees us that there exists a Hadamard diagonalizable graph of order $16k+8$ for each $k\ge 1$ obtained as the product of $K_2\times G$ as long as there is a Hadamard diagonalizable graph of order $8k+4$. More generally, there exists a Hadamard diagonalizable graphs of order $n=2^m(8k+4)$ for each $m, k\ge 1$ given that there is one of order $8k+4$. (See Conjecture \ref{conj26} below.)

\subsection{Order $32$}
The calculation for order $32$ is much more involved, as the search space for each individual matrix grows substantially. In addition, the number of Hadamard matrices of order 32 is far greater---there are 13,710,027 inequivalent normalized Hadamard matrices of order 32. To run the computation, a program was written in \texttt{C++} and used on \texttt{nauty}~\cite{nauty} (for graph isomorphism testing) and \texttt{parallel}~\cite{parallel} (to speed up the computation). The calculation was performed on a server maintained by the Department of Applied Mathematics at Charles University in Prague.  The calculation took 179,736,390 seconds of CPU time, which was about 2 months of real time due to parallel processing.  If Conjecture~\ref{permcols} were true, the calculation would take only 2 days.  The source code, inputs, and outputs can be downloaded from \oururl; this includes all Hadamard matrices stored as strings, all Hadamard diagonalizable graphs stored as \texttt{graph6-strings}, and an additional file that can be used to determine which graphs are associated with which matrix.

There are a total of 10,196 different Hadamard diagonalizable graphs, and unlike smaller orders, many of them are not Cayley graphs. We can partition the Hadamard matrices according to which graphs they diagonalize: the result is 53,420 different equivalence classes. In Figure~\ref{fig:32plot} we mark the distribution of these equivalence classes in log-log scale; each point is an equivalence class.

\begin{figure}[!ht]
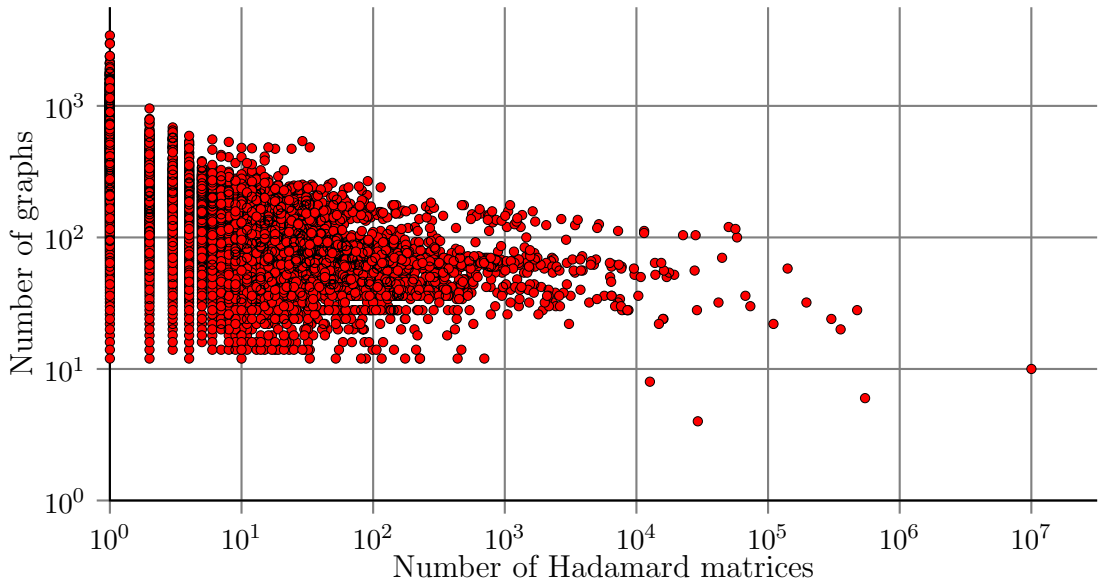

\centering
\picHISTOGRAM
\caption{The sizes of the equivalence classes of Hadamard matrices based on which graphs they diagonalize.}
\label{fig:32plot}
\end{figure}

Here are a few additional notes about the order 32 Hadamard diagonalizable graphs and the equivalence classes; much remains to be explored.
\begin{itemize}
\item The equivalence class corresponding with the fewest number of graphs is associated with only four graphs ($K_{32}$, $K_{16,16}$, $2K_{16}$, $32K_1$); there are 29,270 Hadamard matrices in this group. These Hadamard matrices \emph{only} diagonalize these four graphs---in particular, the four graphs listed here are the only ones which are universally Hadamard diagonalizable for all Hadamard matrices of order $32$.
\item The equivalence class corresponding with the greatest number of graphs is associated with 3,430 graphs, and it has a unique Hadamard matrix in the class. This Hadamard matrix is represented pictorially in Figure~\ref{fig:32matrix}(a).
\item There are 26,064 different equivalence classes which consist of a single Hadamard matrix. The number of graphs associated with these equivalence classes range from $12$ at the low end, up to 3,430 at the high end. If a matrix is associated with more than $956$ Hadamard diagonalizable graphs then it is in an equivalence class of size $1$; on the other hand there does exist an equivalence class associated with $956$ Hadamard diagonalizable graphs which has two non-equivalent Hadamard matrices in the class.
\item The largest equivalence class consists of 10,012,656 Hadamard matrices (out of a possible 13,710,027). There are ten graphs associated with this equivalence class: ($32K_1$, $4K_8$, $2K_{16}$, $K_{32}$, $2K_{8,8}$, $K_{16,16}$, $K_{8,8,8,8}$, $K_2\wr(2K_8)$, $H_{4,4}\wr K_4$, $(K_4\square K_2)\wr (4K_1)$).
\item There are $970$ Hadamard diagonalizable graphs for which each graph is associated with a \emph{unique} equivalence class; moreover for $966$ of these graphs the equivalence class has size $1$. These $966$ graphs are spread among $13$ different Hadamard matrices; $92$ of these graphs are associated with the Hadamard matrix in Figure~\ref{fig:32matrix}(a) and $224$ of these graphs are associated with the Hadamard matrix in Figure~\ref{fig:32matrix}(b). 
\end{itemize}

\begin{figure}[!ht]
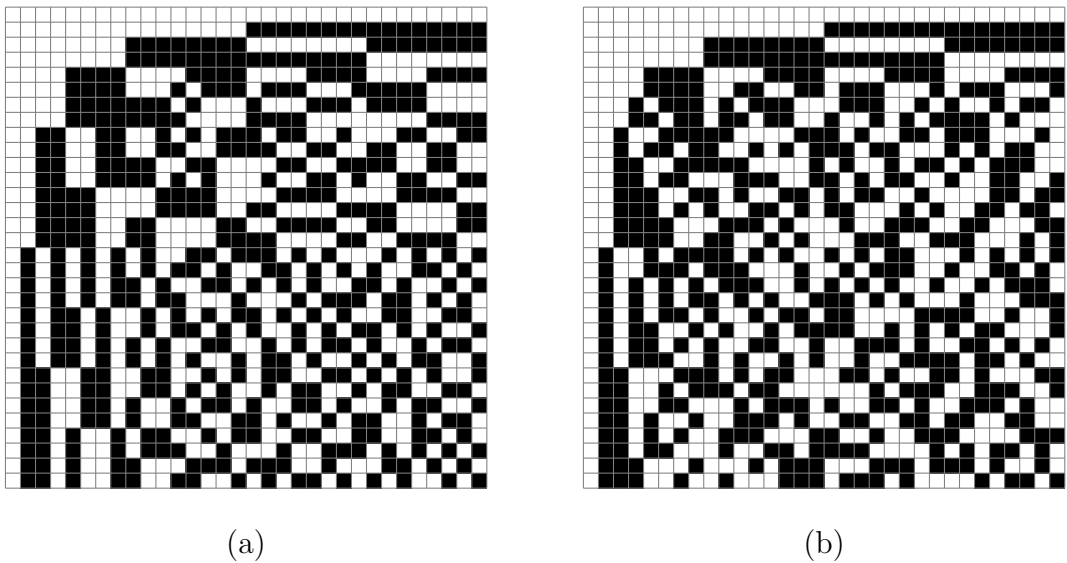

\centering
\begin{tabular}{c@{\hspace{0.5in}}c}
\picLARGE & \picDIVERSE 
\\[10pt]
(a)&(b)
\end{tabular}
\caption{Two Hadamard matrices of order $32$ presented pictorially with white cells corresponding with $1$ and black cells with $-1$. The Hadamard matrix in (a) diagonalizes $3,430$ graphs ($92$ unique to this matrix). The Hadamard matrix in (b) diagonalized $1,684$ graphs ($224$ unique to this matrix).}
\label{fig:32matrix}
\end{figure}

Among the 10,196 Hadamard diagonalizable graphs of order $32$ we have the following data. (Recall that both a graph and its complement are always diagonalized by the same Hadamard matrix.)
\begin{itemize}
\item The graphs are regular, and the degrees of the graphs, denoted $d(G)$, are distributed as follows (we only give information up through degree $15$, the rest follow by symmetry):

\begin{tabular}{|@{\,\,}r@{\,\,}||@{\,\,}c@{\,\,}|@{\,\,}c@{\,\,}|@{\,\,}c@{\,\,}|@{\,\,}c@{\,\,}|@{\,\,}c@{\,\,}|@{\,\,}c@{\,\,}|@{\,\,}c@{\,\,}|@{\,\,}c@{\,\,}|@{\,\,}c@{\,\,}|@{\,\,}c@{\,\,}|@{\,\,}c@{\,\,}|@{\,\,}c@{\,\,}|@{\,\,}c@{\,\,}|@{\,\,}c@{\,\,}|@{\,\,}c@{\,\,}|@{\,\,}c@{\,\,}|} \hline
$k$&0&1&2&3&4&5&6&7&8&9&10&11&12&13&14&15\\ \hline
$|\{G: d(G){=}k\}|$&1&1&1&2&4&6&16&29&56&101&208&343&584&877&1241&1628\\ \hline
\end{tabular}

\item The clique number of a graph, denoted $\nu(G)$, is the size of the largest complete subgraph and its distribution for this class of graphs is as follows:

\begin{tabular}{|@{\,\,}r@{\,\,}||@{\,\,}c@{\,\,}|@{\,\,}c@{\,\,}|@{\,\,}c@{\,\,}|@{\,\,}c@{\,\,}|@{\,\,}c@{\,\,}|@{\,\,}c@{\,\,}|@{\,\,}c@{\,\,}|@{\,\,}c@{\,\,}|@{\,\,}c@{\,\,}|@{\,\,}c@{\,\,}|@{\,\,}c@{\,\,}|@{\,\,}c@{\,\,}|} \hline
$k$&1&2&3&4&5&6&7&8&9&10&16&32\\ \hline
$|\{G:\nu(G){=}k\}|$&1&53&43&4115&1205&1847&443&2435&1&8&44&1\\ \hline
\end{tabular}

Note that taking complements sends cliques to independent sets and so this also gives information about the sizes of maximal independent sets in graphs.
\item There are 10,142 graphs which have girth $3$, $51$ graphs which have girth $4$, $1$ graph which has girth $6$, and $2$ graphs which have no cycles. 
\item There are $54$ disconnected graphs; among the remaining 10,142 connected graphs the diameter of the graph, denoted $\text{diam}(G)$, which is the maximal distance between two vertices is distributed as follows:

\begin{tabular}{|@{\,\,}r@{\,\,}||@{\,\,}c@{\,\,}|@{\,\,}c@{\,\,}|@{\,\,}c@{\,\,}|@{\,\,}c@{\,\,}|@{\,\,}c@{\,\,}|} \hline
$k$&1&2&3&4&5\\ \hline
$|\{G:\text{diam}(G){=}k\}|$&1&9001&1128&11&1\\ \hline
\end{tabular}

\item By Proposition~\ref{prop:regular_even} we know that the eigenvalues for the Laplacian of these graphs consist of even integers. There are 1,228 distinct spectra which are achieved. For $518$ of these graphs the spectrum is unique among these graphs (e.g.\ no other graph from among this list has the same spectrum); the remaining 9,678 graphs each have one or more cospectral mates in the list. The largest cospectral family is for the spectra $\{0^{(1)},12^{(10)},16^{(15)},20^{(6)}\}$ and $\{0^{(1)},12^{(6)},16^{(15)},20^{(10)}\}$
(here exponents represent multiplicity); each family having $528$ distinct graphs with that spectra.

The algebraic connectivity of a graph is the second smallest eigenvalue (counting multiplicity) of the Laplacian matrix of a graph. The algebraic connectivity, $\alpha(G)$, for the graphs are as follows:

\begin{tabular}{|@{\,\,}r@{\,\,}||@{\,\,}c@{\,\,}|@{\,\,}c@{\,\,}|@{\,\,}c@{\,\,}|@{\,\,}c@{\,\,}|@{\,\,}c@{\,\,}|@{\,\,}c@{\,\,}|@{\,\,}c@{\,\,}|@{\,\,}c@{\,\,}|} \hline
$k$&0&2&4&6&8&10&12&14\\ \hline
$|\{G: \alpha(G){=}k\}|$&54&56&398&604&2241&1771&3231&822\\ \hline
\end{tabular}

\begin{tabular}{|@{\,\,}r@{\,\,}||@{\,\,}c@{\,\,}|@{\,\,}c@{\,\,}|@{\,\,}c@{\,\,}|@{\,\,}c@{\,\,}|@{\,\,}c@{\,\,}|@{\,\,}c@{\,\,}|@{\,\,}c@{\,\,}|@{\,\,}c@{\,\,}|@{\,\,}c@{\,\,}|} \hline
$k$&16&18&20&22&24&26&28&30&32\\ \hline
$|\{G: \alpha(G){=}k\}|$&774&122&88&17&12&2&2&1&1\\ \hline
\end{tabular}

\item 4,130 of the graphs are vertex-transitive; $45$ of the graphs are edge-transitive; $38$ of the graphs are distance-regular; $32$ of the graphs are cographs; and $6$ of the graphs are chordal (namely those which are unions of cliques).

\end{itemize}

\section{Concluding remarks}\label{sec:conclusion}
The obstacles to moving forward with larger Hadamard matrices are the size of the computations for any individual Hadamard matrix, combined with a lack of a classification of all Hadamard matrices of order $36$ or above. However, we can run the computation on some known  Hadamard matrices of higher orders and we summarize the computation results in Table~\ref{tab:larger_runs}. 

\begin{table}[h]
\centering
\label{tab:larger_runs}
\begin{tabular}{|l|c|}\hline
Hadamard matrix&Number of H.\ graphs\\ \hline
\texttt{had.40.tpal} & $26$\\ \hline
\texttt{had.40.ttoncheviv} & $26$\\ \hline
\texttt{had.40.twill} & $26$\\ \hline
\texttt{had.48.pal} & $4$\\ \hline
\texttt{had.56.tpal2} & $26$\\ \hline
\texttt{had.56.twll} & $26$\\ \hline
\end{tabular}
\caption{Some Hadamard matrices from Sloane \cite{Sloane} and the number of graphs for which that matrix Hadamard diagonalizes the graph.}
\end{table}

For the three Hadamard matrices of order $40$ the $26$ graphs are the same; similarly, for the two Hadamard matrices of order $56$.  The data, combined with what we know for order $24$ suggests the following.

\begin{conj}\label{conj26}
For $n=16k+8, k= 1, 2, \ldots,$ there are at most $26$ distinct graphs which are Hadamard diagonalizable for some Hadamard matrix of order $n$.
\end{conj}

A proof of this might follow along the lines of that carried out for $n=8k+4$; a disproof would likely come from computations on additional Hadamard matrices, say of order $40$ or $56$, to find additional graphs.

\bigskip
We have seen that there exist Hadamard diagonalizable graphs of order $8k+4$ for all $k<250$ except for the 13 values of $k$, namely, $k=83, 89, 111, 125, 141, 155, 173, 179, 209, 221, 239, 243$ and $245$ 
for each of which it is not known whether there exists a Hadamard matrix of order $8k+4$ as of 2018. Thus we see the existence of Hadamard diagonalizable graphs of various orders. For instance, we know that there are many Hadamard diagonalizable graphs of order 48, as 48 is factored as $2\times 24$ and $4\times 12$ and there are 26 Hadamard diagonalizable graphs of order 24, and 4 for each of order 4 and 12. On the other hand, through the computational search,
the Hadamard matrix of order $48$ from Sloane \cite{Sloane} has few Hadamard diagonalizable graphs.  When we reran the computation using a Hadamard matrix generated by \texttt{SageMath} there were $762$ distinct Hadamard diagonalizable graphs.  Given the lack of classification for Hadamard matrices of order $48$, it is not clear how to determine all Hadamard diagonalizable graphs of order $48$.
\bigskip

Finally, as a by-product, we have the following interesting observation.
We have seen that the Shrikhande graph, say $S$, shown in Table \ref{tab:had24} is the Cayley graph of  $\mathbb{Z}_4^2$ with connecting set $\{\pm(0,1), \pm(1,0), \pm(1,1)\}$. Both $S$ and the $(2,4)$-Hamming graph $H(2,4)=K_4\square K_4$ are strongly regular graph with parameters $(v,k,\lambda, \mu)=(16,6, 2,2)$. They are known as co-spectral graphs. Now as both of them are Hadamard diagonalizable graphs, their Cartesian and lexicographic products and powers are all Hadamard diagonalizable graphs. 

The Hamming graph $H(d,q)$ is isomorphic to the Cartesian product of $d$ copies of $K_q$ (i.e., $H(d,q)\cong K_q^{\square d}$), which is a distance-regular graph of diameter $d$.\footnote{Here a graph $G$ is distance-regular if for any choice of $h, i, j\ge 0$ and any $u,v\in V(G)$ with $d(u,v)=h$, the number of vertices $w\in V(G)$ such that $d(u,w)=i$ and $d(v,w)=j$ is independent of the choice of $u$ and $v$ but depends only on the choice of $h, i$ and $j$.}
The Cartesian product of $l$ copies of $S$ and one copy of Hamming graph $H(d,4)$ is known as a Doob graph $D(l,d)$ of diameter $2l+d$. The Doob graph $D(l,d)$, the Hamming graph $H(2l+d, 4)$, the Cartesian product of $l$ copies of $H(2,4)$ with $H(d,4)$ are cospectral. As a consequence, we state this as the following:

\begin{cor}
The Hamming graphs $H(d, 4)$, $d\ge 1$, and Doob graphs $D(l, d)$, $l, d\ge 1$ are all Hadamard diagonalizable graphs.
\end{cor}

\subsection*{Acknowledgements}
This work was completed in part at the 2019 Graduate Research Workshop in Combinatorics, which was supported in part by NSF grant \#1923238, NSA grant \#H98230-18-1-0017, a generous award from the Combinatorics Foundation, and Simons Foundation Collaboration Grants \#426971 (to M. Ferrara), \#316262 (to S. Hartke) and \#315347 (to J. Martin).


\end{document}